\documentclass[11pt,leqno, a4paper]{amsart}

\usepackage{geometry}                
\usepackage{graphicx}
\usepackage{amssymb}
\usepackage{color}
\usepackage{setspace}
\usepackage{mathtools}
\usepackage{dsfont}
\usepackage{commath}
\usepackage{enumerate}
\usepackage{mathrsfs}
\usepackage{fancyhdr}
\usepackage{amsthm}
\usepackage{scalerel} 
\usepackage{fancyhdr}
\usepackage{bbm}
\usepackage{hyperref}
\usepackage[normalem]{ulem}
\usepackage{tikz}
\usepackage{subcaption} 

\usetikzlibrary{decorations.fractals}

\newtheorem*{rep@theorem}{\rep@title}
\newcommand{\newreptheorem}[2]{%
	\newenvironment{rep#1}[1]{%
		\def\rep@title{#2 \ref{##1}}%
		\begin{rep@theorem}}%
		{\end{rep@theorem}}}

\newtheorem{theorem}{Theorem}[section]

\newtheorem{lemma}[theorem]{Lemma}

\newtheorem{proposition}[theorem]{Proposition}

\newreptheorem{theorem}{Theorem}
\newreptheorem{corollary}{Corollary}

\theoremstyle{definition}

\newtheorem{example}[theorem]{Example}

\theoremstyle{remark}
\newtheorem{remark}[theorem]{Remark}

\numberwithin{equation}{section}

\newcommand{\cA}{\mathcal{A}}

\newcommand{\cB}{\mathcal{B}}

\newcommand{\eye}{\mathds{1}}

\newcommand{\N}{\mathbb{N}}

\renewcommand{\tilde}{\widetilde}

\newcommand{\defeq}{\vcentcolon=}

\def\XXint#1#2#3{{\setbox0=\hbox{$#1{#2#3}{\int}$ }
\vcenter{\hbox{$#2#3$ }}\kern-.6\wd0}}

\title{Families of well approximable measures}
\author{S.~Fairchild, M.~Goering, C.~Wei\ss{}}
\date{\today}

\begin{document}
\begin{abstract}
	We provide an algorithm to approximate a finitely supported discrete measure $\mu$ by a measure $\nu_{N}$ corresponding to a set of $N$ points so that the total variation between $\mu$ and $\nu_N$ has an upper bound. As a consequence if $\mu$ is a (finite or infinitely supported) discrete probability measure on $[0,1]^{d}$ with a sufficient decay rate on the weights of each point, then $\mu$ can be approximated by $\nu_N$ with total variation, and hence star-discrepancy, bounded above by $(\log N) N^{-1}$. Our result improves, in the discrete case, recent work by Aistleitner, Bilyk, and Nikolov who show that for any normalized Borel measure $\mu$, there exist finite sets whose star-discrepancy with respect to $\mu$ is at most $(\log N)^{d-\frac{1}{2}} N^{-1}$. Moreover we close a gap in the literature for discrepancy in the case $d=1$ showing both that Lebesgue is indeed the hardest measure to approximate by finite sets and also that all measures without discrete components have the same order of discrepancy as the Lebesgue measure.
\end{abstract}
\maketitle

\section{Introduction}

In \cite{ABN18}, the authors ask whether the Lebesgue measure is the hardest measure to approximate by finite sets. They guess that the answer is \emph{yes} and justify the conjecture because the Lebesgue measure is spread throughout the entire cube $[0,1]^{d}$ and treats all points the same. In this paper, the main notion of approximation used is the total variation. 

\subsection{Total variation}
For two probability measures $\mu, \nu$ the total variation is defined by 
\begin{equation} \label{e:tvd} 
\rho(\mu;\nu) \defeq \sup_{A \in \cB} |\mu(A) - \nu(A)|,
\end{equation}
where $\cB$ is the class of all Borel sets. The total variation generates the strong topology on the set of all probability measures on $[0,1]^d$. Moreover it is widely used in statistics and has direct connections to other important (statistical) notions such as the Hellinger distance, the Kullback-Leibler divergence, the star-discrepancy and the Wasserstein metric, see \cite{Hel09}, \cite{KS51}, \cite{Nie92}, \cite{Vil09}. Exact calculations of the total variation of two arbitrary probability measures can be difficult if not impossible but for many applications approximations suffice. In this paper, we are interested in giving bounds for the total variation between discrete measures and probability measures associated to a finite set $(x_{i})_{i=1}^N$. By a probability measure associated to a finite set we mean a measure given by
\begin{equation} \label{e:nun}
\nu_{N} = \frac{1}{N} \sum_{i=1}^{N} \delta_{x_{i}},
\end{equation}
where $\delta_{x_{i}}$ denotes the Dirac measure, i.e.
$$ 
\delta_{x_i}(A) := \begin{cases} 1 & \textrm{if} \ x_i \in A\\ 0 & \textrm{else}.  \end{cases}
$$ 
We show that for large classes of discrete probability measures there exist finite sets $(x_i)_{i=1}^N$ so that the total variation is at most $\frac{\log(N)}{N}$. More precisely, one can combine Theorem \ref{t:infdisc} and Example \ref{x:1} to form the following weaker, but easier to understand theorem.
\begin{theorem}\label{t:ex} For every $d \geq 1$ and every $0 < r < 1$, there exists a constant $c_r$ independent of the dimension such that the following holds. For every $N\geq  2$ and every probability measure $\mu$ on $[0,1]^d$ of the form $\mu = \sum_{i=1}^{\infty} \alpha_i \delta_{y_i}$ with $\alpha_{i} \leq r^{i-1} \alpha_{1}$ for all $i \in \mathbb{N}$, there exists a finite set $x_1,\ldots, x_N \in [0,1]^d$ such that $\nu_{N}$ as in \eqref{e:nun} satisfies
	$$\rho\left(\mu; \nu_{N}\right) \leq c_{r} \frac{\log(N)}{N}.$$
The constant $c_r$ is defined explicitly in Example~\ref{x:1} and is monotonically increasing in $r$.
\end{theorem}
The proof of Theorem \ref{t:infdisc} which produces Theorem \ref{t:ex} is composed of two main parts. First, we constructively approximate discrete measures with a finite number of points, and second we use the fact that most points in a discrete probability measure have negligible total mass allowing the finite approximation to suffice.
\begin{remark} \label{r:nonunif}
	The condition $\alpha_{i+1} \leq r^{i} \alpha_{1}$ for $r <1$ should be interpreted as enforcing that the discrete measure is not too close to a uniform distribution over points. Moreover the precise formula for $c_r$ shows that, at least using this method, measures which are closer to being uniformly distributed are harder to approximate. This matches the heuristic argument mentioned above that the Lebesgue measure should be the hardest measure to approximate, since a measure with equal weights centered at a large number of points almost uniformly covering $[0,1]^d$, e.g. a lattice, is \textit{close to} Lebesgue. Example \ref{ex:2} produces a family of measures with faster decay on the weights (hence the measures are further from being uniform) that have an even better approximation by finite sets.
\end{remark}

\subsection{Comparison to other metrics}
If $d_A$ denotes some other (pseudo-)metric on the space of probability measures, often inequalities of the form
\begin{equation} \label{eq:int:1}
d_A(\mu; \nu) \leq h(\rho(\mu;\nu))
\end{equation}
can be proven, where $h$ is an explicitly known function, see \cite{GS02}. For example, this is true for the star-discrepancy $D^{*}_{N}$ and the $p$-Wasserstein metric $W_p$. Thus, Theorem~\ref{t:ex} does not only yield bounds for the total variation distance between the two measures but for any other (pseudo-)metric which satisfies an inequality of the form~\eqref{eq:int:1}. We will focus only on the total variation and star-discrepancy for the remainder of the paper.

 Instead of taking the supremum over all Borel sets only boxes anchored at zero are considered for the star-discrepancy. More precisely, the star-discrepancy between two probability measures $\mu, \nu$ is defined by
$$
D^*_{N}(\mu;\nu) \defeq \sup_{A \in \mathcal{A}} \left| \mu(A) - \nu(A)\right|,
$$
where $\mathcal{A}$ is the set of all half-open axis-parallel boxes contained in $[0,1]^d$ which have one vertex at the origin. Hence, the trivial inequality
$$D^*_{N}(\mu; \nu) \leq \rho(\mu; \nu)$$
holds. Therefore, Theorem~\ref{t:ex} is also true after replacing the total variation by the star-discrepancy. The order of convergence therein, $\tfrac{\log(N)}{N}$, which is independent of the dimension, should be compared to the best known result for general measures. 

\begin{theorem}[Aistleitner, Bilyk, Nikolov, \cite{ABN18}] \label{thm:ABN} For every $d \geq 1$, there exists a constant $c_d$ (depending only on $d$) such that the following holds. For every $N \geq 2$ and every normalized Borel measure $\mu$ on $[0,1]^d$ there exist points $x_1,\ldots,x_N \in [0,1]^d$ such that $\nu_N$ as in \eqref{e:nun} satisfies
	$$D^{*}_{N}\left(\mu; \nu_N \right) \leq c_d \frac{(\log N)^{d-\frac{1}{2}}}{N}.$$ 
Moreover if $(x_n)_{n=1}^\infty$ is a sequence, then the discrepancy bound becomes $\leq c_d N^{-1}(\log N)^{d +\frac{1}{2}}$ for all $N \geq 2$.
\end{theorem}

Theorem~\ref{thm:ABN} improved previous results from \cite{AD14} and \cite{Bec84} where (for finite sets) the exponent of $\log(N)$ was $(3d+1)/2$ and $2d$. In order to confirm the conjecture that Lebesgue measure is the hardest measure to approximate, Theorem~\ref{thm:ABN} would need the exponent of the log reduced by one half to match the classical low-discrepancy bounds for Lebesgue measure. On the other hand, the measures in Theorem~\ref{t:ex} can have strictly better rates of convergence in their discrepancy than the Lebesgue measure when $d\geq 3$, see Examples \ref{x:1} and \ref{ex:2}.

The particular case of the star-discrepancy will be considered in Remark~\ref{r:discr} when we note how the constant $c_r$ in Theorem~\ref{t:ex} can be improved for the star-discrepancy in comparison to the total variation. We will also expand upon the star-discrepancy in detail in Section~\ref{sec:1d}, where we close a gap in the literature when $d =1$ showing Lebesgue measure is the worst measure to approximate by finite sets or sequences, and measures without discrete components have the same optimal order of convergence as Lebesgue measure.

\subsection{Acknowledgment} This research began at the trimester program \textit{Dynamics: Topology and Numbers} at the Hausdorff Research Institute for Mathematics (HIM) in Bonn. The authors would like to thank HIM for bringing them together and for their hospitality. Moreover, the authors would like to thank Stefan Steinerberger for insightful comments on the first draft of this document. The second author was partially supported by PIMS and by FRG DMS-1853993.

\section{Total variation} \label{s:tv}
In this section we consider discrete probability measures on $[0,1]^{d}$. That is, measures of the form 
\begin{equation} \label{e:mu}
\mu = \sum_{i=1}^{\infty} \alpha_{i} \delta_{y_{i}}
\end{equation}  where $\sum_{i} \alpha_{i} = 1$ and $(y_{i}) \subset [0,1]^{d}$. Given sufficient decay of the $(\alpha_{i})$, we provide an algorithm for picking sets $(x_{i}^{N})_{i=1}^{N}$ that produce finite Dirac measures $\nu_{N}$ as in \ref{e:nun} whose total variation distance from $\mu$ has quantifiable decay rate. 

The interested reader should note that the proofs work identically if you replace the stronger notion of total variation measure, with the weaker notion of star-discrepancy in every instance throughout this proof.
\begin{lemma}[Finite to infinite lemma] \label{lem:fin_inf} 
	Let $\mu$ be an infinitely supported discrete probability measure as in \eqref{e:mu}. For each $K \in \N$ define 
	$$ \mu_K = \left(\sum_{i=1}^K \alpha_i\right)^{-1} \sum_{i=1}^K \alpha_i \delta_{y_i}.$$
	If there exists a continuous decreasing function $h: [2,\infty)\to [0,1/2]$ and a constant $c$ so that for all $N\in \N$ there exists $K\geq N$ and a set $(x_{i}^{N})_{i=1}^{N}$ so that $\nu_{N}$ as in \eqref{e:nun} satisfies
	$$\rho\left( \mu_K; \nu_{N} \right) \leq c h(N),$$
	then
	\begin{equation} \label{e:l1}
\rho \left(\mu; \nu_{N} \right) \leq \left(c + 3 \right) h(N).
\end{equation}
\end{lemma}

\begin{proof}
Fix $N \in \N$ and choose $K_{0} \in \mathbb{N}$ such that $\sum_{i=K_{0}+1}^\infty \alpha_i \le h(N)$. By assumption, there exists $K \ge \max\{K_{0},N\}$ and a set $x_1,\ldots,x_N$ with
	$$
	\rho\left(\mu_K; \nu_{N}\right)  \leq c h(N).
	$$
It follows from the triangle inequality that
\begin{align*}
	\rho \left( \mu; \nu_{N}\right) & \leq \rho \left(\mu; \mu_K \right) +\rho(\mu_{K}; \nu_N).
\end{align*}
To bound the first term, $\sum_{i=K+1}^\infty \alpha_i \le h(N)$ ensures for any $A \in \cB([0,1]^{d})$
\begin{align*}
|{\mu}_K(A) - \mu(A)| & = \left|\left(\frac{1}{\sum_{i=1}^K \alpha_i} - 1\right) \sum_{i=1}^K \alpha_i \delta_{y_i}(A) - \sum_{i=K+1}^\infty \alpha_i\delta_{y_i}(A)\right| \\
& \le \frac{h(N)}{1-h(N)} + h(N) \\
& = h(N) \left( \frac{1}{1 - h(N)} + 1 \right).
\end{align*}
As $0\leq h(N) \le \frac{1}{2}$, this verifies \eqref{e:l1}.
\end{proof}
Now we describe an algorithm that, given a measure $\mu$ supported on $k$ points, chooses $N$ points defining $\nu_{N}$ as in \eqref{e:nun} whose total variation distance from $\mu$ is at most $\frac{c_{k}}{N}$ where $c_{k} \le 2 k$. After Propostion~\ref{prop:approx_fin}, we will combine with Lemma~\ref{lem:fin_inf} to get information the total variation distance of infinitely supported discrete measures.

\begin{proposition} \label{prop:approx_fin} Let $\mu$ be probability on $[0,1]^d$ which is supported on no more than $k$ points. Let $N \in \mathbb{N}$. There exists a constant $c_k$ independent of the dimension and a finite set $x^{N}_{1}, \dots, x^{N}_{N} \in [0,1]^d$ so that $\nu_{N}$ as in \eqref{e:nun} satisfies $\rho \left(\mu; \nu_{N} \right) \leq  \frac{c_k}{N}.$ Notably, $c_k \leq 2 k.$
\end{proposition}

\begin{proof}
If $N \leq k-1$ then choose
$$
x_{i}^{N} = \begin{cases} (0,0, \dots, 0) & i \le \lceil \frac{N-1}{2} \rceil \\
(1,1, \dots, 1) & \lceil \frac{N-1}{2} \rceil < i \le N
\end{cases}
$$ 
that is, all points of the finite sequence lie on the origin or on its opposite corner. For these $N$ we have the trivial bound,
$$
\rho(\mu; \nu_{N}) \le 2 \le 2 \cdot  \frac{k-1}{N} \le \frac{c_{k}}{N}.
$$
If $N \geq k$, write $\mu = \sum_{i=1}^k \alpha_i \delta_{y_i}$ with $\sum_{i=1}^k \alpha_i = 1$. For each $\alpha_i$ choose $p_i \in \N$ such that	
\begin{equation} \label{e:nonneg}
0 \leq \alpha_i - \frac{p_i}{N} < \frac{1}{N}.
\end{equation}
Summing the preceding inequality over $i$ ensures there exists $r \in \N$ such that $r < k$ and
	$$\sum_{i=1}^k \alpha_i - \sum_{i=1}^k \frac{p_i}{N} = \frac{r}{N} \iff N - \sum_{i=1}^k p_i = r.$$
Next, we define a finite set $(x_{i}^{N})_{i=1}^{N}$ by first placing $p_{j}$ points on each $y_{j}$ then placing the remaining points at the origin and the opposite corner. More formally, adopting the notation that $\sum_{i=1}^{0} \alpha_{i} = 0$, for $h=1,\ldots, N-r$ choose
	$$x_{h}^{N} = y_{j} \quad \text{whenever} \quad \sum_{i=1}^{j-1}p_i < h \leq \sum_{i=1}^j p_i. 
	$$
To place the remaining $r$ points, define
	$$
	x_{h}^{N} = \begin{cases} (0,0, \dots, 0) & N-r < h \le N- \lceil \frac{r}{2} \rceil \\
		(1,1, \dots, 1) & h > N- \lfloor \frac{r}{2} \rfloor . 
	\end{cases}
	$$
Next, choose some $A \in \cB$. Let $I$ be the index set defined by $i \in I \iff y_{i} \in A$. Then, \eqref{e:nonneg} yields
$$
|\mu(A) - \nu_{N}(A)| \le | \nu_{N} ( \{0\} \cup \{1\})| + \sum_{i \in I} \alpha_{i} - \frac{p_{i}}{N} \le \frac{2r}{N}.
$$
Since $r < k$, $c_{k} = 2k$ suffices.
\end{proof}

We are now ready to prove the main theorem of this section. Examples \ref{x:1} and \ref{ex:2} will illuminate how the explicit upper-bound works.

\begin{theorem} \label{t:infdisc} Let $(r_{k})_{k \in \N}$ be a sequence so that $r_{1} = 1$, $\sum_{k=1}^{\infty} r_{k} < \infty$, and $r_{i} = 0$ implies $r_{j} = 0$ for all $j \ge i$. Moreover suppose there is a constant $c_{0}$ so that for all $M\geq 2$, $\sum_{k=M}^\infty r_k \leq \frac{c_{0}}{M}.$ 

Fix $d \geq 1$ and a discrete probability measure $\mu$ as in \eqref{e:mu} on $[0,1]^{d}$ where $(\alpha_{i})_{i=1}^\infty$ satisfy
\begin{equation} \label{e:sprops1}
\alpha_{k} \le r_{k} \alpha_{1}.
\end{equation}

For all $N\geq 2$ there exists a finite set $(x^{N}_{i})_{i=1}^{N}$ so that $\nu_{N}$ as in \eqref{e:nun} satisfies
\begin{equation} \label{e:tvb}
\rho \left( \mu; \nu_{N} \right) < (6c_0 +3) \frac{ g^{-1} \left( \frac{1}{N \alpha_{1}}  \right)}{ N}
\end{equation}
where $g: [1, \infty) \to [0,\infty)$ is a continuous, strictly decreasing function such that 
\begin{equation} \label{e:gprops}
\begin{cases}
\sum_{i=M}^{\infty} r_{i} \le g(M) & M \in \N_{\ge 2} \\
g(s) \le \frac{c_{0}}{s} & s \in [1,\infty).
\end{cases}
\end{equation}
\end{theorem}
We refer to the function $g$ as the gauge function corresponding to the sequence $(r_{k})_{k \in \N}$.

\begin{proof}[Proof of Thoerem~\ref{t:infdisc}]

Let $g$ be as in the theorem statement. Note, the gauge function $g$ can be chosen to be strictly decreasing on its support since $K \mapsto \sum_{i=K}^{\infty} r_{i}$ is strictly decreasing for all $K \le K_{0}$ and $K_{0}$ is the (possibly infinite) number of non-zero $r_{i}$. 

For $N \in \N$, choose $K_N$ to be the smallest integer so that $K_N \geq g^{-1}\left(\frac{1}{\alpha_1 N}\right)$. In particular, $K_N < g^{-1}\left(\frac{1}{\alpha_1 N}\right) +1.$

Now we consider the normalized Borel measure 
$$
\mu_{K_N} = \frac{1}{\sum_{i=1}^{K_{N}} \alpha_{i}} \sum_{i=1}^{K_{N}} \alpha_{i} \delta_{y_{i}}.
$$
By Proposition \ref{prop:approx_fin}, there exists a set $(x_{i}^{N})_{i=1}^{N}$ such that
\begin{equation} \label{e:discboundg}
\rho \left( \mu_{K_N}, \nu_N \right) \le \frac{2 K_{N}}{N} < 2 \frac{ g^{-1} \left(  \frac{1}{N \alpha_{1}}  \right)+1}{N} \leq 3 \frac{g^{-1}\left(\frac{1}{N\alpha_1}\right)}{N}
\end{equation}
because $g^{-1}(x) \geq 1$.
To apply Lemma~\ref{lem:fin_inf}, we need to find an appropriate function $h:[2,\infty) \to [0, 1/2]$. Consider, 
\begin{equation} \label{e:hdef}
h(N) \defeq (2c_{0} N)^{-1} g^{-1}\left((N \alpha_{1})^{-1}\right).
\end{equation}
To confirm $h(N) \le 1/2$, note \eqref{e:gprops} gives $g(c_{0} N) \leq N^{-1} \leq (N\alpha_1)^{-1}$. Since $g$ is monotonically decreasing this implies $c_{0}N \geq g^{-1}\left(\frac{1}{N\alpha_1}\right),$ verifying $h(N) \le 1/2$. Equation \eqref{e:discboundg} and \eqref{e:hdef} imply the constant $c$ in Lemma \ref{lem:fin_inf} is $6c_{0}$. Thus, Lemma~\ref{lem:fin_inf} verifies \eqref{e:tvb}.
\end{proof}

\begin{remark}  \label{r:discr}
The interested reader may note that nowhere in Lemma \ref{lem:fin_inf}, Proposition \ref{prop:approx_fin}, or Theorem \ref{t:infdisc} do we use the fact that $\mu$ is a measure on $[0,1]^{d}$. All three of the results could go through identically if $(X,\mu)$ were some probability space and $\mu$ is a probability measure. In Proposition \ref{prop:approx_fin}, the points $(0, \dots, 0)$ and $(1, \dots, 1)$ can be replaced with any points in the space $X$. 

However we chose to write the proof this way because it allows for a better bound for the constant term when working with discrepancy. Indeed \eqref{e:nonneg} yields for any $A \in \mathcal{A}$ an axis parallel box,
$$
- \left \lceil \frac{r}{2} \right\rceil \le \mu(A) - \nu(A) = - \nu(\{0\}) + \sum_{i \in I}\alpha_{i} - \frac{p_{i}}{N} \le \left\lfloor \frac{r}{2} \right\rfloor.
$$
The upper and lower bound can potentially be achieved when $I \in \{\emptyset , \{1, \dots, k \} \}$. Finally, since $r < k$, $c_{k} = k/2$ suffices for a star-discrepancy specific version of Proposition \ref{prop:approx_fin}.
\end{remark}

\begin{example} \label{x:1}
Set $r_{k} = r^{k-1}$ for some $0<r< 1$, then for every $N \ge 2$ there exists $(x_{1}^{N})_{i=1}^{N}$ such that the corresponding $\nu_{N}$ as in \eqref{e:nun} satisfies
\begin{equation} \label{e:x1_1}
\rho\left( \mu; \nu_{N} \right) \le c_{r} \frac{ \log(N)}{N}.
\end{equation} 
Indeed, note
\begin{equation} \label{e:x1_2}
\sum_{k=K}^{\infty} r_{k} = \frac{r^{K}}{1-r}.
\end{equation}
So we choose $g(s) \defeq \frac{r^{s}}{1-r}$ for $s \ge 1$ as a strictly decreasing gauge function for the tail of the sequence. We need $g(s) \le \frac{c_{0}}{s}$. We write $c_{0} = \frac{r^{c_{r}}}{1-r}$ and compute $c_{0}$. This means, we require
$$
\frac{r^{s}}{1-r} \le \frac{r^{c_{r}}}{s(1-r)} \iff s r^{s} \le r^{c_{r}}
$$
Note $G(s) \defeq s r^{s}$ is maximized when $s = -1/\ln(r)$, so choosing $c_{r}$ so that $r^{c_{r}} = G(-1/\ln(r))$ yields $c_{0} = \frac{-1}{e \ln(r) (1-r)}$ satisfies $g(s) \le \frac{c_{0}}{s}$. Defining $s_{N} = g^{-1} \left( \frac{1}{\alpha_{1}N} \right)$ implies
$$
\frac{r^{s_{N}}}{1-r} = \frac{1}{N\alpha_{1}}
$$
which is equivalent to
\begin{align*}
s_{N} &= \frac{ \log (\alpha_{1} N) - \log(1-r)}{- \log(r)} \le \frac{ \log(N) - \log(1-r)}{- \log(r)}  \le \tilde{c}_{r} \log(N)
\end{align*}
where
$$
 \tilde{c}_{r}  = \frac{ \log(2) - \log(1-r)}{ - \log(r) \log(2)}.
$$

Applying Theorem \ref{t:infdisc} verifies \eqref{e:x1_1} where $c_r =(6 c_{0} + 3) \tilde{c}_{r}$. Since $c_{0}$ (which despite the notation depends on $r$) and $\tilde{c}_{r}$ are both monotonically increasing for $r \in (0,1)$, the same can be said about $c_{r}$.
\end{example}
 
The next example will use a fact which we recall here for the reader's convenience. See for instance, \cite[5.1.20]{AS48}.
\begin{proposition} \label{p:fact}
For any $a > 0$, the following holds
$$
\int_{a}^{\infty} \frac{e^{-t}}{t} \mathrm{d} t < e^{-a} \log \left( 1 + \frac{1}{a} \right).
$$ 
\end{proposition}

We emphasize that in the next example, the discrete measures have total variation distance that, under the algorithm of Proposition \ref{prop:approx_fin} converges faster than the Lebesgue discrepancy, even when $d = 2$ (see Remark \ref{r:nonunif}). In particular, the star-discrepancy for any measure in this family of converges faster than the Lebesgue.
\begin{example} \label{ex:2}
Let $r_{k} = \frac{r^{e^{k}}}{r^{e}}$ where $0<r < \frac{1}{2}$. Then for each $N$, there exists a set $(x_{i}^{N})_{i=1}^{N}$ and a constant $c_{r}$ so that the associated measure $\nu_{N}$ as in \eqref{e:nun} satisfies
\begin{equation} \label{e:x2}
\rho \left( \mu; \nu_{N} \right) \le c_{r} \frac{ \log \left( \frac{\log(N)}{|\log(r)|} \right)}{N}.
\end{equation}
Indeed, $g(m) = r^{e^m}$ satisfies the assumptions of Theorem~\ref{t:infdisc} since, using Proposition \ref{p:fact}
\begin{align*}
\sum_{k=m+1}^{\infty} r^{e^{k}} &\le \int_{m}^{\infty} 	r^{e^{s}} \mathrm{d} s = |\log(r)| \int_{e^{m}}^{\infty} \frac{e^{-|\log(r)|t}}{|\log(r)|t} \mathrm{d} t =  \int_{|\log(r)|e^{m}}^{\infty} \frac{e^{-u}}{u} \mathrm{d} u \\
& < e^{- |\log(r)| e^{m}} \log \left( 1 + \frac{1}{|\log(r)|e^{m}} \right) \\
& \le e^{-|\log(r)|e^{m}} = r^{e^{m}},
\end{align*}
The penultimate line comes from the fact that $r < \frac{1}{2} < e^{(e^{-1}-1)}$, and thus
$$1 + \frac{1}{|\log(r)|e^m} < e$$
for all $m \geq 1$. Moreover, $g(s) \le \frac{c_{0}}{s}$ with $c_{0} = r^{e}$.
Define $s_{N} = g^{-1} \left( \frac{1}{\alpha_{1}N} \right)$. That is,
$$
r^{e^{s_{N}}} = \frac{1}{\alpha_{1}N} \iff \log(r) e^{s_{N}} = - \log(\alpha_{1} N) \iff s_{N} = \log \left( \frac{\log(\alpha_{1}N)}{|\log(r)|} \right) \le \log \left( \frac{\log(N)}{|\log(r)|} \right).
$$
Applying Theorem \ref{t:infdisc} verifies \eqref{e:x2} with $c_{r} = (6 r^{e} +3)$.
\end{example}

The last example is intended to demonstrate that this technique of approximating discrete measures in general cannot be the optimal way to do so. In this example, it only yields something equivalent to the trivial bound for total variation distance. One reason to expect this method is not sharp is that it completely ignores the location of the points, which should be very important especially in higher dimensions.
\begin{example} \label{ex:3} Let $r_k = 1/k^2$. Then the same method only guarantees that there exists a set $(x_{i}^{N})_{i=1}^{N}$ so that the associated $\nu_{N}$ as in \eqref{e:nun} satisfies
\begin{equation} \label{e:x3}
\rho \left( \mu; \nu_{N} \right) \le 9.
\end{equation}
Indeed,
$$\sum_{k=m+1}^\infty \frac{1}{k^2} \leq \int_{m}^\infty \frac{1}{s^2} \mathrm{d}s = \frac{1}{m} =:g(m).$$
Then $K_N = g^{-1}((\alpha_1N)^{-1}) = \alpha_1 N \le N$. Applying Theorem \ref{t:infdisc} gives the bound in \eqref{e:x3}.
\end{example}

\section{Discrepancy in one dimension} \label{sec:1d}
For the $d$-dimensional Lebesgue measure $\lambda_d$ it is conjectured that there exists a constant $c_d$ dependent only on the dimension such that for every finite atomic measure $\nu_N$ centered at $N$ points, $x_1,\ldots,x_N$, the inequality
$$D^*_{N}(\lambda_d; \nu_{N}) \geq c_d \frac{(\log N)^{d-1}}{N}$$
holds infinitely often. In other words, the optimal order of approximation of the Lebesgue measure by a finite atomic measure is conjectured to be $N^{-1}(\log N)^{d-1}$. In fact, this is standard knowledge for dimensions one \cite{Nie92}. In dimension two, this is known by the work of Schmidt, \cite{Sch72}. 

Sets that achieve the optimal order of approximation are called low-discrepancy point sets. Similarly, sequences with the conjectured optimal order of convergence $N^{-1}(\log N)^d$ are called low-discrepancy sequences (for the Lebesgue measure). There are essentially three classical families of low-discrepancy sequences for the one-dimensional Lebesgue measure: Kronecker sequences, digital sequences, and Halton sequences. In dimension one, further classes of examples have recently been found, e.g. \cite{Car12}, \cite{Wei19}. A discussion of the multi-dimensional picture can be found in \cite{Nie92}. 

For arbitrary normalized Borel measures $\mu$, Theorem~\ref{thm:ABN} yields the best known order of approximation. If the measure $\mu$ has a non-vanishing continuous component then the method of Roth \cite{Rot54}, over the orthogonal functions can be applied in order to find lower bounds for the star-discrepancy. This has been conducted in \cite{Che85}. The lower bounds which are obtained in this way are the same as those for the Lebesgue measure. This can be regarded as another indication that the Lebesgue measure is particularly hard to approximate by finite sets, and that the results of Theorem~\ref{thm:ABN} should be able to be improved to show the conjecture that Lebesgue measure is indeed the most difficult measure to approximate by finite sets.

Though more is not known for arbitrary Borel measures in any dimension, we do know that the bounds in Theorem~\ref{thm:ABN} are not optimal when $d= 1$. The reason we have better bounds when $d=1$ is that the ordering of $[0,1]$ is equivalent to the ordering (by inclusion) of axis-parallel boxes in $[0,1]$ which contain the origin. Hence this case can be treated by generalizing the arguments given in \cite{HM72a}, \cite{HM72b}.\footnote{Hlawka and M\"uck concentrated on deriving inequalities of Koksma-Hlawka type in their papers and hence did not work out the details regarding approximation of measures. Moreover they made a Lipschitz continuity assumption.}

\begin{remark} \label{r:ldt}
	Recall that in dimension $d=1$, Lebesgue's decomposition theorem states any Borel measure $\mu$ can be written as 
	$$\mu = \mu_{ac} + \mu_{d} + \mu_{cs},$$
	where $\mu_{ac}$ is absolutely continuous with respect to the Lebesgue measure, that is $\mu_{ac}$ is zero on sets of Lebesgue measure zero, $\mu_d$ is a discrete measure, that is, it is zero on the complement of some countable set, and $\mu_{cs}$ is continuous singular, that is, $\mu_{cs}$ is zero on the complement of some set $B$ of Lesbesgue measure zero but assigns no weight to any countable set of points. For more details we refer the reader, e.g. to \cite{HS75}, Chapter V. 
\end{remark}

\begin{theorem} \label{thm:01}
	Fix $\mu$ a normalized Borel measure on $[0,1]$. 
	\begin{enumerate}
		\item For all $N\in \N$, there exists a finite set $(x_i)_{i=1}^N$ such that $\nu_N$ as in \eqref{e:nun} satisfies 
		$$D^*_{N}\left(\mu; \nu_N \right) \leq \frac{c}{N}$$
		with $c$ a constant independent of $N$.
		
		\item Moreover, there exists a sequence $(x_{k})_{k \in \N} \in [0,1]$ such that for all $N \in \mathbb{N}$ and $\nu_N$ as in \eqref{e:nun} we have
		$$D^*_{N}\left( \mu; \nu_N \right) \leq c \frac{\log(N)}{N}$$
		with a constant $c$ independent of $N$.
		
		\item Suppose $\mu$ is a Borel measure with no point masses. That is, $\mu = \mu_{ac} + \mu_{cs}$. Then \begin{equation}\label{eq:lowerbound}
		D_N^*\left(\mu ; \nu_N \right) \geq \frac{1}{2N}
		\end{equation}
		for any finite set $(x_i)_{i=1}^N$. Moreover, there exists a constant $c$ so that for infinitely many $N\in \N$, 
		\begin{equation}\label{eq:schmidt} D^*_{N}\left( \mu; \nu_N \right) \geq c \frac{\log N}{N}
		\end{equation}
		holds for any sequence $(x_k)_{k\in \N}$ with $\nu_N$ defined as in \eqref{e:nun}.
	\end{enumerate}
\end{theorem}

\begin{remark} \label{r:schmidt}
	Part (3) of Theorem~\ref{thm:01}, combined with (1) and (2) states that in the case where $\mu$ has no point masses, the optimal rate of convergence is $O(N^{-1})$ and $O(N^{-1}\log(N))$ for sets and sequences respectively. Note that when $\mu = \lambda_{1}$ is the Lebesgue measure, it was already known this is the best rate for convergence for point sets \cite[Theorem~2.6]{Nie92} and sequences \cite{Sch72}.
\end{remark}

In order to prove Theorem~\ref{thm:01}, we will make use of the following lemma which constructs finite sets and/or sequences for which one can compare the discrepancy with respect to Lebesgue and to an arbitrary Borel measure.
\begin{lemma}\label{lem:d=1}
	Let $\mu$ be a probability measure on $[0,1]$. If $(v_k)_{k= 1}^{M} \subseteq [0,1]$ is a finite or countably infinite set of points, then there exists $(x_{k})_{k = 1}^{M}$ so that for all $2 \le N \le M$ and $\nu_N$ as in \ref{e:nun},
	\begin{equation}\label{eq:comparediscrepancy} D_N^*\left( \mu; \nu_N\right) \leq D_N^*\left( \lambda_1; \frac{1}{N}\sum_{i=1}^N \delta_{v_i}\right).
	\end{equation}
Moreover, if $\mu$ has no point masses, that is $\mu_{d}([0,1]) = 0$ (see Remark \ref{r:ldt}), then in Equation~\ref{eq:comparediscrepancy} we in fact have equality.
\end{lemma} 
The strategy of the proof is to use a cumulative distribution type function for $\mu$ to create an injective function from $\cA \to \cA$ (recall $\cA$ is the set of half-open intervals contained in $[0,1]$ containing $0$) that can be used to pull back a set or sequence $(v_{k})$ with given Lebesgue discrepancy to find a set or sequence $(x_{k})$ that has the same $\mu$ discrepancy. When $\mu_{d}$ is zero, the distribution function creates a bijection not just an injection. This fails in higher-dimensions because axis parallel boxes do not have a well-ordering that also respects the geometry. Namely different boxes can see different points in a different order.

\begin{proof}[Proof of Lemma~\ref{lem:d=1}] Note that the function
$$
f(a) \defeq \mu \left( [ 0,a) \right)
$$
is non-decreasing and continuous from the left.

Let $(v_{k})_{k \geq 1}$ be an arbitrary finite or infinite set. For each $k \geq 1$ define $x_{k}$ so that
\begin{equation} \label{e:01}
\begin{cases}
f(a) \le v_{k} \qquad \forall a \le x_{k} \\
f(a) > v_{k} \qquad \forall a > x_{k}.
\end{cases}
\end{equation}
This can be done since $f$ is non-decreasing and continuous from the left. Fix $N \in \N_{\geq 2}$ and an interval $[0, b)$.

Consider $c \defeq f(b)$. Due to the monotonicity and one-sided continuity of $f$ one of the following holds:
$$
f^{-1} \left( \{ c \} \right) = \begin{cases} \{b\} \\ (b - \delta_{1}, b + \delta_{2}]  \qquad \text{for some } \delta_{1}, \delta_{2} \ge 0. \end{cases}
$$
The first case occurs when $f$ is strictly increasing at $b$, and the second case occurs if $f$ is constant on a sub-interval containing $b$. In either case, define $b_{0} = \max \{ f^{-1} \left( f(b)\right) \}$. 

\textit{Claim 1:}  The discrepancy of $(x_k)_{k=1}^N$ with respect to $\mu$ over the interval $[0,b_0)$ is the same as the Lebesgue discrepancy of $(v_k)_{k=1}^N$ over $[0,c)$. That is
\begin{equation} \label{e:02}
\left| \frac{\# \{ x_{k} < b_{0} : k \le N \} }{N} - \mu([0,b_{0})) \right| = \left| \frac{ \# \{ v_{k} < c : k \le N \}}{N} - \lambda_1( [0, c)) \right|.
\end{equation}
Indeed, since $f(b_{0}) = c = \lambda_1([0, c))$ we have
$\mu([0,b_{0})) = \lambda_1([0,c)).$ Moreover, if $x_{k} < b_{0}$ then $f(b_{0}) > v_{k}$ by \eqref{e:01}. In particular, $v_{k} < c$. On the other hand, if $v_{k} < c$ then
$$
f(x_{k}) \le v_{k} < c = f(b_{0}),
$$
so the non-decreasing nature of $f$ combined with the strict inequality forces $x_{k} < b_{0}$. This verifies Claim 1.

\textit{Claim 2:} The discrepancy of $(x_k)_{k=1}^N$ for $\mu$ over the intervals $[0,b)$ and $[0,b_0)$ are equal. That is
\begin{equation} \label{e:03}
\left| \frac{\# \{ x_{k} < b_{0} : k \le N \} }{N} - \mu([0,b_{0})) \right|  = \left| \frac{\# \{ x_{k} < b : k \le N \} }{N} - \mu([0,b)) \right|.
\end{equation}
Indeed, if $b= b_0$, this statement is trivial. If $b_0 > b$, then we still have $f(b_0) = f(b)$, so we only need to show that $(x_k)_{k\leq N} \cap [b, b_0) = \emptyset$. Since $f$ is constant on the interval $[b,b_0)$, by $\eqref{e:01}$, either $x_k = b_0$ which is outside $[b,b_0)$, or $x_k \leq b- \delta_1$, so $(x_{k})_{k \le N} \cap [b, b_{0}) = \emptyset $ as desired.

Claim $1$ and Claim $2$ ensure that for every interval $[0,b)$ associated to $\mu$, the interval $[0,f(b))$ has the same Lebesgue discrepancy. Taking the supremum over all $b$ ensures
\begin{equation} \label{e:1}
D_{N}^{*}\left( \mu; \nu_N\right) = \sup_{b}  \left|\sum_{i=1}^N \frac{ \eye_{[0,f(b))}(v_{i})}{N} - \lambda_{1}([0,f(b))  \right| \le D_{N}^{*} \left(\lambda_1 ; \frac{1}{N} \sum_{i=1}^{N} \delta_{v_{i}} \right),
\end{equation}
verifying \eqref{eq:comparediscrepancy}.

When $\mu = \mu_{ac} + \mu_{sc}$ has no discrete part, then $f$ is in fact continuous. Therefore $f$ maps onto $[0,1)$. Hence taking the supremum over $[0,f(b))$ is now equivalent to taking the supremum over $[0,c)$ forcing an equality in \eqref{e:1}. 
\end{proof}

We now utilize Lemma~\ref{lem:d=1} to prove Theorem~\ref{thm:01}. 
\begin{proof}[Proof of Theorem~\ref{thm:01}]
\begin{enumerate}
	\item For $k = 1,\ldots, N$, set $v_k = \frac{2k-1}{2N}$. Let $x_{k}$ be defined as in \eqref{e:01}.  Then by Theorem~2.7 of \cite{Nie92} combined with Lemma~\ref{lem:d=1},
	$$ D_{N}^{*}\left(\mu; \nu_N \right) \le D_{N}^{*}\left(\lambda_1; \frac{1}{N}\sum_{i=1}^N \delta_{v_i} \right)  = \frac{1}{N}.$$
	\item Now let $(v_k)_{k\in \N}$ be any low-discrepancy sequence with respect to Lebesgue measure and let $(x_{k})_{k \in \N}$ be defined as in \eqref{e:01}. Then by Lemma~\ref{lem:d=1}, 
	$$
D_{N}^{*}\left( \mu ; \nu_N \right) \le D_{N}^{*}\left(\lambda_1 ;\frac{1}{N}\sum_{i=1}^N \delta_{v_i} \right)\leq c \frac{\log(N)}{N}.
$$
	\item Finally, suppose $\mu = \mu_{ac} + \mu_{sc}$. Combining Lemma~\ref{lem:d=1} with the results of Schmidt and Niederreiter (see Remark \ref{r:schmidt}) verifies \eqref{eq:lowerbound} and \eqref{eq:schmidt}.
\end{enumerate}
\end{proof}
Thus for an arbitrary probability measure $\mu$ on $[0,1]$, there exist sequences whose $\mu$ discrepancy converges to zero at least as fast as in the Lebesgue case. Moreover when $\mu$ has no discrete part, the rate from the Lebesgue measure is in fact optimal. The following example shows that in the discrete case there exist sequences whose discrepancy converge strictly faster than in the Lebesgue case.

\begin{example}\label{ex:1/2measure} Let $\delta_{y}$ denote the Dirac measure centered at $y$ and consider the Borel measure $\mu := \frac{1}{2} \delta_{0} + \frac{1}{2} \delta_{0.5}$. Then the sequence $(x_{k})_{k \in \N}$ defined by $x_{2k+1} = 0$ and $x_{2k} = 1/2$ has 
	\begin{align*}
	D_N^*\left( \mu; \nu_N \right) = \begin{cases}  \frac{1}{2N} & N \ \textrm{odd} \\ 0 & N \ \textrm{even}. \end{cases}
	\end{align*}
\end{example}

\bibliographystyle{alpha}
\bibdata{references}
\bibliography{references}

\begin{thebibliography}{HM72b}

\bibitem[ABN18]{ABN18}
C.~Aistleitner, D.~Bilyk, and A.~Nikolov.
\newblock Tusn\'{a}dy\'{}s problem, the transference principle, and non-uniform
  qmc sampling.
\newblock {\em Monte Carlo and Quasi-Monte Carlo Methods - MCQMC 2016},
  241:169--180, 2018.

\bibitem[AD14]{AD14}
C.~Aistleitner and J.~Dick.
\newblock Low-discrepancy point sets for non-uniform measures.
\newblock {\em Acta Arith.}, 163 (4):345--369, 2014.

\bibitem[AS48]{AS48}
Milton Abramowitz and Irene~A Stegun.
\newblock {\em Handbook of mathematical functions with formulas, graphs, and
  mathematical tables}, volume~55.
\newblock US Government printing office, 1948.

\bibitem[Bec84]{Bec84}
J.~Beck.
\newblock Some upper bounds in the theory of irregularities of distribution.
\newblock {\em Acta Arith.}, 43 (2):115--130, 1984.

\bibitem[Car12]{Car12}
I.~Carbone.
\newblock Discrepancy of $ls$-sequences of partitions and points.
\newblock {\em Ann. Mat. Pura Appl.}, 191:819--844, 2012.

\bibitem[Che85]{Che85}
W.~Chen.
\newblock On irregularities of distribution and approximate evaluation of
  certain functions.
\newblock {\em Quart. J. Math. Oxford Ser. (2)}, 36(142):173--182, 1985.

\bibitem[GS02]{GS02}
A.~Gibbs and F.~Su.
\newblock On choosing and bounding probabilty metrics.
\newblock {\em International Statistical Review}, 70:419--435, 2002.

\bibitem[Hel09]{Hel09}
E.~Hellinger.
\newblock Neue begr\"undung der theorie quadratischer formen von unendlich
  vielen ver\"anderlichen.
\newblock {\em J Reine Angew Math.}, 136:210--271, 1909.

\bibitem[HM72a]{HM72a}
E.~Hlawka and R.~M\"uck.
\newblock A transformation of equidistributed sequences.
\newblock {\em Applications of number theory to numerical analysis (Proc.
  Sympos., Univ. Montr\'eal, Montreal, Que., 1971)}, pages 371--388, 1972.

\bibitem[HM72b]{HM72b}
E.~Hlawka and R.~M\"uck.
\newblock \"uber eine transformation von gleichverteilten folgen.
\newblock {\em II.Computing (Arch. Elektron. Rechnen)}, 9:127--138, 1972.

\bibitem[HS75]{HS75}
E.~Hewitt and K.~Stromberg.
\newblock {\em Real and Abstract Analysis: A Modern Treatment of the Theory of
  Functions of a Real Variable}.
\newblock Graduate Texts in Mathematics, 25, Springer, 1975.

\bibitem[KL51]{KS51}
S.~Kullback and R.~Leibler.
\newblock On information and sufficiency.
\newblock {\em Annals of Mathematical Statistics}, 22 (1):79--86, 1951.

\bibitem[Nie92]{Nie92}
H.~Niederreiter.
\newblock {\em Random Number Generation and Quasi-Monte Carlo Methods}.
\newblock Number 63 in CBMS-NSF Series in Applied Mathematics, SIAM,
  Philadelphia, 1992.

\bibitem[Rot54]{Rot54}
K.~F. Roth.
\newblock On irregularities of distribution.
\newblock {\em Mathematika}, 1:73--79, 1954.

\bibitem[Sch72]{Sch72}
W.~M. Schmidt.
\newblock Irregularities of distribution vii.
\newblock {\em Acta Arith.}, 21:45--50, 1972.

\bibitem[Vil09]{Vil09}
C.~Villani.
\newblock {\em Optimal Transport, old and new}.
\newblock Springer, 2009.

\bibitem[Wei19]{Wei19}
C.~Wei\ss{}.
\newblock Interval exchange transformations and low-discrepancy.
\newblock {\em Ann. Mat. Pura Appl.}, 198 (2):399--410, 2019.

\end{thebibliography}
	
\end{document}